\documentclass{amsart}
\usepackage{amsfonts}
\usepackage{amsmath,amssymb}
\usepackage{amsthm}
\usepackage{amscd}
\usepackage{graphics}
\usepackage{graphicx}
\theoremstyle{definition}{
	\newtheorem{Def}{{\rm Definition}}
	\newtheorem{Ex}{{\rm Example}}
	
	\newtheorem{Prob}{{\rm Problem}}
}
\theoremstyle{plain}
{
	
	\newtheorem{Prop}{Proposition}
	\newtheorem{Thm}{Theorem}
	\newtheorem{MainThm}{Main Theorem}

}

\begin{document}
	\title[Manifolds admitting certain very explicit special-generic-like maps]{Restrictions on Manifolds admitting certain explicit special-generic-like maps and construction of maps with the manifolds}
	\author{Naoki Kitazawa}
	\address{Institute of Mathematics for Industry, Kyushu University, 744 Motooka, Nishi-ku Fukuoka 819-0395, Japan\\
		TEL (Office): +81-92-802-4402 \\
		FAX (Office): +81-92-802-4405 \\
	}
	\email{n-kitazawa@imi.kyushu-u.ac.jp}
	\urladdr{https://naokikitazawa.github.io/NaokiKitazawa.html}
	
	\begin{abstract}
		{\it Special-generic-like} maps or {\it SGL} maps are introduced by the author motivated by observing and investigating algebraic topological or differential topological properties of manifolds via nice smooth maps whose codimensions are negative. The present paper says that manifolds admitting certain very explicit SGL maps are topologically restricted strongly and this also constructs these maps with the manifolds explicitly.
		
		Morse functions with exactly two singular points on spheres or functions in Reeb's theorem and canonical projections of naturally embedded spheres in Euclidean spaces are generalized as special generic maps. Their nice global structures motivate us to study such maps and the manifolds. Manifolds represented as connected sums of the product of spheres and similar ones admit such maps in considerable cases. The topologies and the differentiable structures of the manifolds of the domains of such maps are also strongly restricted: thanks to Saeki and Sakuma, followed by Nishioka, Wrazidlo and the author. Respecting their nice global structures, these maps are generalized and this yields SGL maps.


		
	\end{abstract}
	
	
	\maketitle
	\section{Introduction.}
	\label{sec:1}
	\subsection{History and motivation.}
{\it Special-generic-like} maps or {\it SGL} maps are introduced by the author recently in \cite{kitazawa8, kitazawa9}. What is our motivation?

Morse functions with exactly two singular points, or functions in Reeb's theoorem and canonical projections of spheres naturally embedded in Euclidean spaces are generalized as special generic maps. Theory of Morse functions have been strong in geometry of manifolds. Their {\it singular} points, defined as points in the manifolds of the domains where the ranks of the differentials drop, have information on homology groups and some homotopy of the manifolds.
Reeb's theorem characterize spheres topologically (except in $4$-dimensional cases). We can say that {\it special generic} maps, and more generally, {\it fold} maps have been introduced and studied and that nice algebraic topological or differential topological theory has developed.

For example, manifolds represented as connected sums of the products of spheres admit such maps and manifolds being similar in some senses admit special generic maps under considerable situations.
On the other hand, the definition gives strong restrictions on maps and the manifolds and the topologies and the differentiable structures of them are strongly restricted. This is due to studies of Saeki and Sakuma mainly in the 1990s, followed by ones of Nishioka, Wrazidlo and the author after the 2010s.

Furthermore, special generic have very nice global structures. Respecting these structures, {\it special-generic-like} maps or {\it SGL} maps are introduced by the author. We are expecting that the class of such maps can cover wider classes of manifolds of the domains and that we can understand geometric properties of the maps and manifolds: especially, algebraic topological ones and differential topological ones.

\subsection{Terminologies, notions and notation.}

For elementary algebraic topology, see \cite{hatcher} for example. We adopt terminologies and notation in such books or ones seeming to be familiar to us.
We explain about some of the theory when we need. For example, fundamental classes of compact, connected and oriented manifolds and Poincar\'e duality.

Most of our topological spaces have the structures of some cell complexes whose dimension are finite. We can define the dimension of of such a space $X$ uniquely and let it be denoted by $\dim X$. Polyhedra are regarded as specific cases of cell complexes. A topological manifold is of the class of such spaces. A smooth manifold is regarded as a polyhedron, an object in the PL category, or the PL structure in a well-known canonical way. This defined the structure of a PL manifold as the object in the category uniquely. Remember that PL category and the piecewise smooth category are shown to be equivalent. Topological manifolds whose dimensions are at most $3$ have the unique structures of polyhedra. They are also the PL manifolds. They are also differentiable and smooth manifolds. For each of these topological manifolds, the differentiable structures are unique. Topological spaces homeomorphic to polyhedra whose dimensions are at most $2$ have unique PL structures.
This is due to \cite{moise} for example.
For a differentiable map $c:X \rightarrow Y$ between differentiable manifolds, a {\it singular point} $x \in X$ is defined as a point where the rank of the differential is smaller than the minimum between the dimensions $\dim X$ and $\dim Y$. The {\it singular set} of the map, denoted by $S(c)$, is the set of all singular points of $c$.

Let $k$ be a positive integer. ${\mathbb{R}}^k$ denotes the k-dimensional Euclidean space, endowed with the natural differentiable structure and the natural structure of a Riemannian manifold by the standard Euclidean metric.
$||x|| \geq 0$ means the metric between $x \in {\mathbb{R}}^k$ and the
 origin $0 \in {\mathbb{R}}^k$ under the metric. $\mathbb{R}:={\mathbb{R}}^1$ and let $\mathbb{Z} \subset \mathbb{Q} \subset \mathbb{R}$ denote the ring of all integers and that of all rational numbers. $S^k:=\{x \in {\mathbb{R}}^{k+1}\mid ||x||=1\}$ is the {\it $k$-dimensional unit sphere} for $k \geq 0$. This is a smooth compact submanifold with no boundary. It is connected for $k \geq 1$. It is a two-point set with the discrete topology for $k=0$. $D^k:=\{x \in {\mathbb{R}}^{k} \mid ||x|| \leq 1\}$ is the {\it k-dimensional unit disk} for $k \geq 1$. This is a smooth compact and connected submanifold. Their topologies are also easily understood. 

We define a {\it diffeomorphism} between smooth manifolds is a smooth map having no singular points and being a homeomorphism. A {\it diffeomorphism on a manifold} is defined as a diffeomorphism from the (smooth) manifold to itself. Two manifolds are defined to be {\it diffeomorphic} if and only if a diffeomorphism between these two exists. This naturally defines an equivalence relation on the family of all smooth manifolds. We add that their corners are eliminated in a canonical way and that this canonical way gives a smooth manifold with no corner. Furthermore, if we consider such eliminations to a fixed manifold, then we always have manifolds which are mutually diffeomorphic.
We can define mutually {\it PL homeomorphic manifolds} via piecewise smooth homeomorphisms similarly. PL homeomorphisms are for specific cases. However, we do not need deep understanding on PL maps and piecewise smooth maps, which are general maps in these categories, essentially. 

The {\it diffeomorphism group} of a manifold is defined as the space consisting of all diffeomorphisms on the manifold and endowed with the so-called {\it Whitney $C^{\infty}$ topology}. It is a (topological) group. 
The {\it Whitney $C^{\infty}$ topology} on the space of all smooth maps between given two smooth manifolds are a natural and important topology. The space and its subspaces are important in the theory of singularity theory of smooth maps and differential topology of manifolds. See \cite{golubitskyguillemin} for example.
\subsection{Short exposition on special generic maps and SGL maps.}
For theory of Morse functions, see \cite{milnor1, milnor2}.
A {\it special generic} map is defined as a smooth map which is locally projections or represented as the product maps of Morse functions obtained by natural heights of unit disks and the identity maps on suitable disks. 

The image of a special generic map is a smoothly immersed manifold whose dimension is same as that of the space of the target. In the interior of the manifold, it is represented as a projection such that the preimages are diffeomorphic to a unit sphere. Around the boundary, it is locally represented as the product maps of Morse functions obtained by natural heights of unit disks and the identity maps on suitable disks. 
A {\it special-generic-like} map or an {\it SGL} map and an FSGL map, a specific one, are defined by replacing the manifold of the preimage of the projection by a general smooth compact manifold and the functions for the local product maps around the boundary by general functions of some suitable classes. 

\subsection{Main Theorems and the content of our paper.}

\begin{MainThm}
\label{mthm:1}

	  We have an FSGL map $f:M \rightarrow {\mathbb{R}}^n$
	  on some $m$-dimensional so-called {\it spin} closed and simply-connected manifold $M$ enjoying the following properties for any integer satisfying $m \geq 5$ and any integer satisfying $l \geq 2$.
	  \begin{enumerate}
	  	\item The image is a smoothly embedded manifold represented as a boundary connected sum of two copies of $S^{n-1} \times D^1$ where the boundary connected sum is considered in the smooth category.
	   	\item The 2nd integral homology group of $M$ is free and of rank $l$ in the case $m=5$ and free and of rank $l-2$ in the case $m \geq 6$.
	   	\item The 3rd integral homology group of $M$ is free. It is of rank $l$ in the case $m=5$. It is of rank $4$ in the case $m=6$. It is of rank $2$ in the case $m \geq 7$.
	  	\item The $j$-th integral homology group is the trivial group for $j \neq 0,2,3,m-3,m-2,m$.
	  	\item In the case $m \neq 5$, the cup product for any ordered pair of an element of $H^{j_1}(M;\mathbb{Z})$ and one of $H^{j_2}(M;\mathbb{Z})$ is the zero element for $j_1,j_2 \in \{2,3\}$.
	  \end{enumerate}
  \end{MainThm}
\begin{MainThm}	 
	\label{mthm:2}
	In Main Theorem \ref{mthm:1}, as specific cases, if $(m,n)=(5,3), (6,4)$, then $M$ is a manifold represented as a connected sum of finitely many copies of $S^2 \times S^{m-2}$ or $S^3 \times S^3$ where we do not need to care about the category.
\end{MainThm}

This is a new result on new explicit observations on maps in \cite{kitazawa9, kitazawa10}. Main Theorem \ref{mthm:1} constructs maps explicitly on manifolds satisfying some topological conditions. Main Theorem \ref{mthm:2} shows strong restrictions on these manifolds. 

In the next section, we review special generic maps and SGL maps more rigorously and systematically. The third section is for Main Theorems. Main Theorem \ref{mthm:1} is presented again later in more rigorous and stronger manners. The fourth section is for new problems on our study. \\
\ \\
\noindent {\bf Conflict of interest.} \\
The author is supported by the project JSPS KAKENHI Grant Number JP22K18267 "Visualizing twists in data through monodromy" (Principal Investigator: Osamu Saeki) as a member. \\
\ \\
{\bf Data availability.} \\
Data essentially supporting our study are all here.

	\section{Preliminaries.}
	We review important terminologies, notions, notation and properties.



\subsection{Smooth bundles and linear bundles.}
		We define a {\it smooth} bundle as a bundle whose fiber is a smooth manifold and whose structure group is (some subgroup of) the diffeomorphism group.
	The class of {\it linear} bundles is important. This is a bundle whose fiber is a Euclidean space, a unit disk, or a unit sphere, and whose structure group consists of linear transformations: here linear transformations can be defined naturally and the groups give .
	
	Precise explanation on general theory of bundles is omitted and see \cite{steenrod} for example. For linear bundles, see \cite{milnorstasheff} for example. 
	
	\subsection{Fold maps and special generic maps.}
	\begin{Def}
		\label{def:1}
		A smooth map $c:X \rightarrow Y$ between two smooth manifolds with no boundaries is said to be a {\it fold} map if at each singular point $p \in X$ we have suitable local coordinates around $p$ (and $f(p)$) and some integer $0 \leq i(p) \leq \frac{m-n+1}{2}$ by them we can represent $c$ as $(x_1,\cdots,x_{\dim X}) \rightarrow (x_1,\cdots,x_{\dim Y-1},{\Sigma}_{j=1}^{\dim X-\dim Y-i(p)+1} {x_{\dim Y+j-1}}^2-{\Sigma}_{j=1}^{\dim X-\dim Y-i(p)+1} {x_{\dim Y+i(p)+j-1}}^2)$.
		If $i(p)$ is chosen as $0$ for each singular point $p$, then $f$ is said to be a {\it special generic} map. 
	\end{Def}
\begin{Ex}
	\label{ex:1}
	We define integers satisfying $k \geq 1$, $k_1,k_2 \geq 1$ and $k=k_1+k_2$. We can define a canonical projection of a unit sphere by a map mapping a point $(x_1,x_2) \in S^k \subset {\mathbb{R}}^{k+1}={\mathbb{R}}^{k_1} \times {\mathbb{R}}^{k_2}$ to $x_1 \in {\mathbb{R}}^{k_1}$. This is shown to be a special generic map and this is a kind of exercises on smooth manifolds and maps, Morse functions and the singularity theory of differentiable maps. 
\end{Ex}	 
	 Since the 1990s, manifolds admitting special generic maps have been actively studied by Saeki and Sakuma. Related results are in \cite{saeki1, saeki2, saekisakuma1, saekisakuma2}, Following these studies, Nishioka and Wrazidlo have obtained  \cite{nishioka, wrazidlo1, wrazidlo2}.
	They have revealed restrictions on the differentiable structures of the homotopy spheres, some elementary manifolds such as ones in Example \ref{ex:1}, which is presented later. Restrictions on the homology groups have been also studied. As a pioneer, the author launched explicit systematic studies on the cohomology rings of the manifolds in \cite{kitazawa1, kitazawa2, kitazawa3, kitazawa4, kitazawa5, kitazawa6, kitazawa7, kitazawa8, kitazawa11} for example.

	Other than canonical projections of unit spheres, we present simplest special generic maps.
	\begin{Ex}
		\label{ex:2}
		Let $l$ be an arbitrary positive integer and $m \geq n \geq 2$ integers. We give an integer $1 \leq n_j \leq n-1$ for each integer $1 \leq j \leq l$. We take a connected sum of $l>0$ manifolds considered in the smooth category where the $j$-th manifold is $S^{n_j} \times S^{m-n_j}$ and have a smooth manifold $M_0$. We have a special generic map $f_0:M_0 \rightarrow {\mathbb{R}}^n$ in such a way that the image is represented as a boundary connected sum of $l$ manifolds considered in the smooth category with the $j$-th manifold being diffeomorphic to $S^{n_j} \times D^{n-n_j}$ 
		
	\end{Ex}

\subsection{Special-generic-like maps or SGL maps.}

A {\it height function of a unit sphere} is a function defined as a map mapping $x=(x_1,x_2) \in S^k \subset {\mathbb{R}}^{k+1}=\mathbb{R} \times {\mathbb{R}}^k$ to $x_1$. We can say that this is generalized to canonical projections of unit spheres.
A {\it height function of a unit disk} is defined by restricting the height function of the unit sphere to the preimage of the natural half-space $\mathbb{R} \bigcap \{t \geq 0 \mid t \in \mathbb{R}\}$.

\begin{Def}
	\label{def:2}
	Let $m \geq n \geq 1$ be integers.
	A {\it special-generic-like} map or an SGL map is a smooth map $f:M \rightarrow N$ on an $m$-dimensional closed and connected manifold $M$ into an $n$-dimensional smooth manifold $N$ which has no boundary enjoying the following properties.
	\begin{enumerate}
		\item \label{def:2.1}
	The image is a smooth immersion ${\bar{f}}_N:\bar{N} \rightarrow N$ of an $n$-dimensional compact and connected manifold $\bar{N}$ into an $n$-dimensional smooth manifold $N$ which has no boundary.
	 
	 	\item \label{def:2.2}
We have a smooth surjection $q_f:M \rightarrow \bar{N}$ enjoying the relation $f={\bar{f}}_N \circ q_f$. Furthermore, we can define $q_f$ as a map whose restriction to the singular set defines a diffeomorphism onto the boundary $\partial \bar{N}$.
	 	\item \label{def:2.3}
 We have some small collar neighborhood $N(\partial \bar{N})$ for the boundary $\partial \bar{N} \subset \bar{N}$ and the composition of the restriction of $q_f$ to the preimage with the canonical projection to the boundary is the projection of some smooth bundle over $\partial \bar{N}$. We call such a bundle a {\it boundary smooth bundle of $f$}.
\item \label{def:2.4}
 On the collar neighborhood $N(\partial \bar{N})$ and the preimage ${q_f}^{-1}(N(\partial \bar{N}))$, it is locally represented as some smooth map represented as the product map of the following two smooth maps for suitable local coordinates around each point $p$ of $\partial \bar{N} \subset N(\partial \bar{N})$. 
\begin{enumerate}
\item \label{def:2.4.1}
 A smooth function ${\tilde{f}}_p:E_p \rightarrow \mathbb{R}$ on some ($m-n+1$)-dimensional smooth compact and connected manifold. 
\begin{enumerate}
\item 
\label{def:2.4.1.1}
 The image is represented as a closed interval $[a_p,b_p]$.
\item
\label{def:2.4.1.2}
 Either the preimage of $a_p$ or $b_p$ and the boundary $\partial E_p$ coincide. If $a_p$ can be chosen here, then $b_p$ is the maximum and the preimage of $t$ contains some singular points of the function if and only if $t=b_p$ and the preimage containing singular points is a polyhedron whose dimension is at most $m-n$. If $b_p$ can be chosen here, then $a_p$ is the minimum and the preimage of $t$ contains some singular points of the function if and only if $t=b_p$ and the preimage containing singular points is a polyhedron whose dimension is at most $m-n$.
\item \label{def:2.4.1.3}
  $E_p$ is diffeomorphic to a fixed manifold $E_{C}$ where $E_{C}$ depends on a connected component $C \ni p$ of the boundary $\partial \bar{N}$. 
\end{enumerate}
Furthermore, the manifold of the domain of this function is regarded as a fiber of some boundary smooth bundle just before. We call this function a {\it product component boundary function} or {\it PCB function} of $f$. 
\item
\label{def:2.4.2}
 The identity map on a small smoothly embedded copy of the unit disk $D^{\dim \partial \bar{N}} \subset \partial \bar{N}$ containing $p$ in the interior.
\end{enumerate}
 
	 	\item
\label{def:2.5}
 The restriction of $q_f$ to the preimage ${q_f}^{-1}(\bar{N}-{\rm Int}\ N(\partial \bar{N}))$ is the projection of some smooth bundle whose fiber is
diffeomorphic to the boundary $\partial E_p$ of the previous manifold. Furthermore, $\partial E_p$ is diffeomorphic to a fixed manifold $F_f$. We call this bundle an {\it internal smooth bundle of $f$}.
	 \end{enumerate}
 As a specific case, if the local functions ${\tilde{f}}_p:E_p \rightarrow \mathbb{R}$ or PCB functions of $f$ can be always chosen as smooth functions represented as the compositions of fold maps, then $f$ is said to be of {\it type F} or an {\it $FSGL$} map. 
	
\end{Def}


\begin{Prop}
Special generic maps are seen as FSGL maps. This had been shown first in \cite{saeki2} long before our notions were introduced. Furthermore, boundary smooth bundles of these maps are always linear according to the theory.
\end{Prop}
\begin{Ex}
	\label{ex:3}
	Special generic maps are also simply generalized special generic maps, first introduced in \cite{kitazawa8}. Simply generalized special generic maps are also FSGL map. In Example \ref{ex:2}, the maps are constructed as ones such that the boundary smooth bundles and the interior smooth bundles of the maps are trivial.
	
	A simply generalized special generic map is defined as an FSGL map such that the PCB function is always represented as the composition of the projection of some trivial smooth bundle whose fiber is diffeomorphic to the product of finitely many smooth manifolds each of which is diffeomorphic to some unit sphere over a unit disk with a height function where suitable local coordinates are considered.
	In short, a smooth map of the class is locally represented as the projection or the product map of a so-called {\it Morse-Bott} function of a good class and the identity map on a disk. This Morse-function also respects local Morse-Bott functions for so-called {\it moment} maps on so-called {\it symplectic toric} manifolds of dimension $2k$ into ${\mathbb{R}}^k$.
	
\end{Ex}
For related theory on symplectic toric manifolds and moment maps, see \cite{buchstaberpanov, delzant} for example. \cite{kitazawa9, kitazawa10} say that our SGLs are introduced to cover classes of manifolds of the domains seeming to be difficult to obtain by special generic maps. Symplectic toric manifolds admit no special generic maps in considerable cases. For example so-called $k$-dimensional projective space ${\mathbb{C}P}^k$, which is also a 2k-dimensional closed and simply-connected smooth manifold, is an important example, according to \cite{kitazawa5}.

\section{On Main Theorems.}
\subsection{Known results on classifications of maifolds admitting special generic maps}
  We define {\it homotopy sphere} as a smooth manifold homeomorphic to a unit sphere whose dimension is not $0$. 
A {\it standard} sphere is a homotopy sphere being diffeomorphic to some unit sphere. An {\it exotic sphere} is one which is not. It is well-known that $4$-dimensional exotic spheres are still unknown. Except these cases, homotopy spheres are known to be PL homeomorphic to unit spheres where the canonically defined PL manifolds are considered for the smooth manifolds. Under this rule, $4$-dimensional exotic spheres are known to be not PL homeomorphic to standard spheres.

By applying these terminologies on homotopy spheres, we present existing results on classifications of special generic maps closely related to our arguments.

\begin{Thm}[\cite{kitazawa5, nishioka,saeki1,saeki2} etc.]
	\label{thm:1}
	\begin{enumerate}
		\item
		\label{thm:1.1} \cite{burletderham, furuyaporto, saeki1, saeki2}
		Let $m$ be an arbitrary integer greater than or equal to $2$.
		A necessary and sufficient condition for an $m$-dimensional closed and simply-connected manifold $M$ to admit a special generic map $f:M \rightarrow {\mathbb{R}}^2$ is that $M$ is a homotopy sphere which is not a $4$-dimensional exotic sphere. 
			A necessary and sufficient condition for an $m$-dimensional closed and connected manifold $M$ to admit a special generic map $f:M \rightarrow {\mathbb{R}}^2$ is that $M$ is represented as a connected sum of finitely many manifolds in the following.
			\begin{enumerate}
				\item A homotopy sphere which is not a $4$-dimensional exotic sphere.
				\item The total space of a smooth bundle over the circle $S^1$ whose fiber is a homotopy sphere which is not a $4$-dimensional exotic sphere.
				\end{enumerate} 
		\item
		\label{thm:1.2} \cite{saeki1, saeki2}
		Let $m$ be an arbitrary integer greater than or equal to $4$.
	If an $m$-dimensional closed and simply-connected manifold $M$ admits a special generic map $f:M \rightarrow {\mathbb{R}}^3$, then $M$ is as follows.
		\begin{enumerate}
			\item \label{thm:1.2.1} A homotopy sphere which is not a $4$-dimensional exotic sphere.
			\item \label{thm:1.2.2} A manifold
			represented as a connected sum of smooth manifolds taken in the smooth category where each manifold here is the total space of a smooth bundle over $S^2$ whose fiber is either of the following two.
			\begin{enumerate}
				\item An {\rm (}$m-2${\rm )}-dimensional homotopy sphere where $m \neq 6$.
				\item A $4$-dimensional standard sphere.
			\end{enumerate}
		\end{enumerate}
		In the case $m=4,5,6$, this gives a necessary and sufficient condition. In the case, a fiber of each bundle can be taken as an {\rm (}$m-2${\rm)}-dimensional standard sphere and the total spaces of the bundles can be regarded as the total spaces of linear bundles.
	    Furthermore, in the cases {\rm (}\ref{thm:1.1}{\rm )} and {\rm (}\ref{thm:1.2}{\rm )},
		the manifold $M$ not being a homotopy sphere admits a special generic map as in Example \ref{ex:2} where the boundary smooth bundles or the internal smooth bundles may not be trivial. More explicitly, let $n_j=2$ in Example \ref{ex:1}.
		\item {\rm [}\cite{nishioka}etc.{\rm]}
		\label{thm:1.3} 
		Let $m$ be an arbitrary integer greater than or equal to $5$.
		If an $m$-dimensional closed and simply-connected manifold $M$ admits a special generic map $f:M \rightarrow {\mathbb{R}}^4$, then the 2nd integral homology group of $M$ is free.
			\item \cite{nishioka}
		\label{thm:1.4} In the case $m=5$ in {\rm (}\ref{thm:1.3}{\rm )}, this gives a necessary and sufficient condition. We remember the classification of $5$-dimensional closed and simply-connected manifolds in the smooth category, the PL, or equivalently, the piecewise smooth category, and the topology category, which are same, in \cite{barden}. We have a list of manifolds admitting special generic maps and a property similar to one in {\rm (}\ref{thm:1.3}{\rm )}.
\item \cite{kitazawa5}
\label{thm:1.5}
In the case $m=6$ in {\rm (}\ref{thm:1.3}{\rm )}, we have a similar list of manifolds admitting special generic maps and a property similar to one in {\rm (}\ref{thm:1.3}{\rm )} where the manifold in {\rm (}\ref{thm:1.2.2}{\rm )} is replaced by one represented as a connected sum of finitely many manifolds each of which is diffeomorphic to $S^2 \times S^4$ or $S^3 \times S^3$ in the list. The connected sums are essentially same in the topology category, piecewise smooth category, the PL category, and the topology category. For classifications of $6$-dimensional closed and simply-connected manifolds, see \cite{jupp, wall1, wall2, zhubr1, zhubr2} for example.
	\end{enumerate}
\end{Thm}

\subsection{Fundamental classes of connected and compact (oriented) manifolds.}
We explain about fundamental classes of connected and compact (oriented) manifolds. For our exposition here, see \cite{hatcher} again for example.

Let $A$ be a commutative ring having the unique identity element different from the zero element and let $1_A$ denote the identity element. $A$ can be considered to be the canonically obtained module over $A$ and in this case, $1_A$ and $-1_A$ are generators of $A$.
For a compact and connected oriented manifold $X$, $H_{\dim X}(X, \partial X; A)$ is isomorphic to $A$ as a module over $A$. We can give a generator respecting the orientation. Related to this, we do not need orientations in the case $A:=\mathbb{Z}/2\mathbb{Z}$ or the commutative ring of order $2$.

For a manifold $Y$, consider an embedding $i_X:X \rightarrow Y$ with suitable conditions being assumed respecting the category. For example, in the smooth category, this is assumed to be smooth and in the PL or the piecewise smooth category, this is assumed to be piecewise smooth. Furthermore, $X$ is embedded {\it properly}. This means that the boundary $\partial X$ is embedded in the boundary and the interior is embedded in the interior ${\rm Int}\ X$. In the smooth category, $X$ is assumed to be embedded in a so-called generic way in general.  
If $a \in H_j(Y, \partial Y; A)$ is realized as the value of the homomorphism ${i_{X}}_{\ast}:H_{\dim X}(X, \partial X; A) \rightarrow H_{\dim X}(Y, \partial Y; A)$ induced canonically from the embedding at the fundamental class of $X$, which is an element of $H_{\dim X}(X, \partial X; A)$, then $a$ is {\it represented} by the submanifold $i_X(X)$.

We remark on a "generic way" in smooth embeddings. In short they satisfy conditions on so-called "transversality". 
This is important in several scenes of our paper.
For our smooth embedding $i_X:X \rightarrow Y$, we consider an embedding satisfying the condition that 
the dimension of the intersection of the image of the differential $d {i_X}_p$ of the embedding $i_X:X \rightarrow Y$ at each point $p \in \partial X$ and the tangent space at the value $i_X(p) \in \partial Y$ is $\dim X+\dim \partial Y-\dim Y=\dim X-1$. This comes from fundamental and important notions in singularity theory of differentiable maps and differential topology. For these notions and arguments, see \cite{golubitskyguillemin} again.

Poincar\'e duality (theorem) or the so-called intersection theory for a compact, connected and oriented manifold is important. We can define {\it Poincar\'e duals} to elements of $H_j(X;A)$ and $H^{j}(X;A)$ as elements of $H^{\dim X-j}(X,\partial X;A)$ and $H_{\dim X-j}(X,\partial X;A)$ uniquely respectively.    We can also define {\it Poincar\'e duals} to elements of $H_j(X,\partial X;A)$ and $H^{j}(X, \partial X;A)$ as elements of $H^{\dim X-j}(X; A)$ and $H_{\dim X-j}(X;A)$ uniquely respectively.

\subsection{Proofs of Main Theorems and related observations and examples.}
Hereafter we encounter some of topological theory of $3$-dimensional manifolds such as a {\it Heegaard splitting} of a $3$-dimensional closed and connected manifold. This is a decomposition of such a manifold into two copies of some manifold represented as a boundary connected sum of finitely many copies of $S^1 \times D^2$ along surfaces of the boundaries or a so-called {\it Heegaard surface}.

See \cite{hempel} for example. However, we do not need to understand such theory well.

For the {\it $j$-th Stiefel-Whitney classes} of linear bundles including tangent bundles, we omit precise exposition. They are defined uniquely as the elements of cohomology groups of degrees $j$ where coefficient ring is $\mathbb{Z}/2\mathbb{Z}$, the commutative ring of order $2$. For the {\it $j$-th Pontrjagin classes} of linear bundles including tangent bundles, we also omit precise exposition. They are defined uniquely as the elements of integral cohomology groups of degrees $4j$.
See \cite{milnorstasheff}. The {\it $j$-th Stiefel-Whitney classes} and the {\it $j$-th Pontrjagin classes} of smooth manifolds are defined as those of the tangent bundles. {\it Spin} manifolds are smooth manifolds whose 2nd Stiefel-Whitney classes are the zero elements.
\begin{Thm}[Main Theorem \ref{mthm:1}]
	\label{thm:2}
	For any integer $m \geq 5$, any integer $l \geq 2$ and any integer $i$ satisfying $i \geq 2$ or $i=0$, we have an FSGL map $f:M \rightarrow {\mathbb{R}}^n$ on some $m$-dimensional so-called {\it spin} closed and simply-connected manifold $M$ enjoying the following properties.
	\begin{enumerate}
		\item The image is a smoothly embedded manifold represented as a boundary connected sum of two copies of $S^{n-1} \times D^1$ where the boundary connected sum is considered in the smooth category.
		\item The 2nd integral homology group of $M$ is free and of rank $l$ in the case $m=5$ and free and of rank $l-2$ in the case $m \geq 6$.
		\item The 3rd integral homology group of $M$ is free. It is of rank $l$ in the case $m=5$. It is of rank $4$ in the case $m=6$. It is of rank $2$ in the case $m \geq 7$.
		\item The $j$-th integral homology group is the trivial group for $j \neq 0,2,3,m-3,m-2,m$.
		\item In the case $m>5$, the cup product for any ordered pair of an element of $H^{j_1}(M;\mathbb{Z})$ and an element of $H^{j_2}(M;\mathbb{Z})$ is the zero element for $j_1,j_2 \in \{2,3\}$.
		\item The $a$-th Stiefel-Whitney class and the $a$-th Pontrjagin class of $M$ are always the zero elements for any positive integer $a$.
	\end{enumerate}
Furthermore, in the case $l>2$, for the resulting maps $f:=f_{i_1}:M:=M_{i_1} \rightarrow {\mathbb{R}}^n$ and $f:=f_{i_2}:M:=M_{i_2} \rightarrow {\mathbb{R}}^n$, we cannot choose any pair $(\Phi:M_{i_1} \rightarrow M_{i_2},\phi:{\mathbb{R}}^n \rightarrow {\mathbb{R}}^n)$ of diffeomorphisms enjoying the relation $\phi \circ f_{i_1}=f_{i_2} \circ \Phi$ for any pair $(i_1,i_2)$ of distinct integers $i=i_1$ and $i=i_2$.
\end{Thm}
\begin{proof}
	General construction of an FSGL map on a closed and simply-connected manifold such that "$m$ in Definition \ref{def:2}" is greater than or equal to $3$ and that $F_f$ is a closed, connected and orientable surface of genus $g \geq 1$ there is essentially same as one in the proof of Theorem 3 (Main Theorm 2) of \cite{kitazawa10} and our proof of main result of \cite{kitazawa0}.

Hereafter, as assumed, we consider the case where the dimension $m$ of the manifold of the domain is at least $5$.\\
	\ \\
	STEP 1 Construction of PCB functions. A review of our proof of Theorem 3 of \cite{kitazawa10}. \\
	\ \\
	We construct our PCB function ${\tilde{f}}_p:E_p \rightarrow \mathbb{R}$, denoted by ${\tilde{f}}_{S^2,g,p}$ for the genus $g \geq 1$. This is represented as the composition of some local smooth map ${\tilde{f}}_{S^2,g}:{S^2}_{(g+1)} \times D^1 \rightarrow D^1 \times D^1=[-1,1] \times [-1,1] \subset {\mathbb{R}}^2$ with a height function on the unit disk where the height function is naturally scaled. More precisely, the function is scaled to a function on the disk which is centered at the origin and whose radius is $\frac{1}{2}$. We can construct the manifold ${S^2}_{(g+1)}$ and the map ${\tilde{f}}_{S^2,g}:{S^2}_{(g+1)} \times D^1 \rightarrow D^1 \times D^1=[-1,1] \times [-1,1] \subset {\mathbb{R}}^2$ as follows.
	\begin{itemize}
		\item ${S^2}_{(g+1)}$ is a compact and connected surface obtained by removing the interiors of $g+1$ smoothly and disjointly embedded copies of the disk $D^2$ in a copy of $S^2$.
		\item If we restrict the map to the interior of the manifold, then it is a fold map into the interior of the product $D^1 \times D^1 \subset {\mathbb{R}}^2$.
		\item If we restrict the map to ${S^2}_{(g+1)} \times \{t\}$ for each $t \in D^1$, then the image is $[-1,1] \times \{t\}$ and by composing the projection to $[-1,1]$, we have a Morse function enjoying the following properties.
		\begin{itemize}
			\item The image is [-1,1].
			\item The preimage of $\{-1\}$ and the exactly $g$ boundary connected components coincide.
			\item The preimage of $\{1\}$ and the exactly one boundary connected component coincide.
		\item If we restrict the map to ${S^2}_{(g+1)} \times \{t\}$ for each $t \in D^1-\{0\}$, then it is a Morse function with exactly $g$ singular points such that at distinct singular points the values are mutually distinct and that the singular points are all in the interior of the surface of the domain.
		\item If we restrict the map to ${S^2}_{(g+1)} \times \{0\}$, then it is a Morse function with exactly $g$ singular points and the values there are always $0$. Furthermore, the singular points are all in the interior of the surface of the domain.
		\end{itemize}
	\end{itemize}
	We explain about the topology of ${S^2}_{(g+1)} \times D^1$. Let $g=1$. In this case, this is, after the corner is eliminated, diffeomorphic to $D^2 \times S^1$. Assume that for $g:=g_0 \geq 1$, this is diffeomorphic to a manifold represented as a boundary connected sum of $g_0$ copies of $D^2 \times S^1$ (after the corner is eliminated). Here remember that we do not care about the category since the dimensions we consider are (at most) $3$. However, as in the original expocition, we can do similarly in higher dimensional cases. We investigate this for $g:=g_0+1$. The manifold we consider is obtained by removing the interior of a small regular neighborhood of a closed interval embedded properly in the manifold in the case $g=g_0$. In other words, the boundary of the closed interval is embedded in the boundary and the interior is embedded in the interior. 
	Furthermore, the embedding is taken as an embedding smoothly homotopic to a constant map whose values are always one point in the boundary $\partial {S^{2}}_{(g_0+1)} \times [-1,1]$ of the manifold ${S^{2}}_{(g_0+1)} \times [-1,1]$ for the case $g=g_0$. We can also choose the smooth homotopy $e:D^1 \times [0,1] \rightarrow {S^{2}}_{(g_0+1)} \times [0,1]$ enjoying the following properties.
	\begin{itemize}
    \item The restriction to $D^1 \times [0,1)$ is a smooth embedding. 
    \item The restrictions to each $D^1 \times \{t^{\prime}\}$ for $t^{\prime} \in [0,1)$ are also smooth embeddings. 
    \end{itemize}
    The properties says that the embedding of the closed interval is "unknotted" (in the smooth category).
	We can understand that our manifold is diffeomorphic to a manifold represented as a boundary connected sum of $g:=g_0+1$ copies of $D^2 \times S^1$.
	
	We construct our PCB function ${\tilde{f}}_p:E_p \rightarrow \mathbb{R}$, denoted by ${\tilde{f}}_{S^2,0,p}$ for the genus $g=0$ as a height function on a copy of the $2$-dimensional unit disk $D^2$. \\
	\ \\	
	STEP 2 General construction of a natural simplest FSGL map. \\
	\ \\
	We construct a desired map by using these PCB functions. 
	For a manifold ${\bar{N}}^{\prime}$ represented as a boundary connected of at least two copies of $S^{n-1} \times D^1$ taken in the smooth category
	we have an FSGL map $f:M \rightarrow N:={\mathbb{R}}^n$ on some closed and connected manifold $M$ such that ${{\bar{N}}}^{\prime}=\bar{N}$ in Definition \ref{def:2}. We can easily embed smoothly and naturally ${{\bar{N}}}^{\prime} \subset N$. Hereafter, we abuse notation in Definition \ref{def:2} in our construction.
	
	We can construct a simplest example by considering trivial bundles as an internal smooth bundle and a boundary smooth bundle and gluing them by some isomorphism between the trivial bundles defined canonically on the boundaries.
	
	More precisely, we glue by the product map of the identity map on the base space and some diffeomorphism on the fiber. In the most specific and natural case, we glue them by the product map of the identity map on the base space and the identity map on the fiber. Note that here we can consider natural identifications between the base spaces of the bundles and the fibers and we consider these identifications.
	   
	 This is a fundamental and natural method and related remarks are also in \cite{kitazawa10} for example. \cite{kitazawa10} also reviews shortly that \cite{kitazawa11} studies the structures and special properties of special generic maps and the manifolds obtained in such ways as most simplest examples. \\
	\ \\
	STEP 3 Construction of an FSGL map into ${\mathbb{R}}^n$ like one in STEP 2 on some closed and simply-connected manifold $M$ of dimension $m=n+2$ whose 2nd integral homology group is free. A review of our proof of Main Theorem 5 of \cite{kitazawa10}. \\
	\ \\
	We show the construction of a desired map as presented in STEP 2 on a simply-connected one whose 2nd integral homology group is free. 
	Let the number of the copies of $S^{n-1} \times D^1$ we need to have ${{\bar{N}}}^{\prime}$ denoted by $l_N \geq 1$.
	Let ${\{P_j\}}_{j=1}^{l_N}$ be any sequence of $3$-dimensional closed, connected and orientable manifolds.
	We can define a family ${\{S_j\}}_{j=0}^{l_N}$ of $l_N+1$ connected components of the boundary $\partial {{\bar{N}}}^{\prime} \subset {{\bar{N}}}^{\prime}$ so that ${\bar{f}}_{N}(S_{0})$ is the only one connected component which is the boundary of the unique non-compact connected component of the closure of the complement of the image ${\bar{f}}_N({{\bar{N}}}^{\prime})$.
	We can have an FSGL map $f:M \rightarrow N$ like one in STEP 2 and choose a family $\{C_j\}$ of $l_N$ mutually disjoint closed interval smoothly and properly embedded in ${{\bar{N}}}^{\prime}=\bar{N}$ enjoying the following properties.
	\begin{itemize}
		\item The boundary of $C_j$ consists of exactly two points and one is embedded in $S_0$ and the other is in $S_j$. 
		\item The interior of $C_j$ is embedded in the interior of ${{\bar{N}}}^{\prime}$.
		\item $C_j$ is embedded satisfying the condition on the transversality.
		\item The preimage ${q_f}^{-1}(C_j)$ is diffeomorphic to $P_j$. This is due to the existence of a Heegaard splitting of any $3$-dimensional closed, connected and orientable manifold. 
	\end{itemize}
Here, we follow our proof of Main Theorem 5 of \cite{kitazawa10}. 
The singular set is, by the construction, of dimension $n-1+1=n$. We also have $n+1<m=n+2$. A smoothly embedded circle in $M$ is, smoothly isotoped to a circle in ${q_f}^{-1}(\bar{N}-{\rm Int}\ N(\partial \bar{N}))$. 
The total space of the internal smooth bundle 
${q_f}^{-1}(\bar{N}-{\rm Int}\ N(\partial \bar{N}))$ is the product of the base space and its fiber, diffeomorphic to a closed, connected and orientble surface $F_f$. We can see that
the fundamental group of ${q_f}^{-1}(\bar{N}-{\rm Int}\ N(\partial \bar{N}))$ is isomorphic to the fundamental group of the surface of the fiber or the preimage. Remember also that $\bar{N}$ collapses to $\bar{N}-{\rm Int}\ N(\partial \bar{N})$ and that the base space $\bar{N}-{\rm Int}\ N(\partial \bar{N})$ is simply-connected. 

Hereafter, we consider the case where some $P_j$ is a $3$-dimensional sphere. We can see that the smoothly embedded circle considered here is shown to be null-homotopic. We can see that $M$ is simply-connected.

Remember that the singular set is of dimension $n$ and we also have $n+2=m$. Assume that the 2nd integral homology group of $M$ is not free. Then there exists an element which is not the zero element and which is of a finite order. It is represented by a $2$-dimensional sphere ${S^2}_{\rm T}$ smoothly embedded in $M$. FOr this remember also that $M$ is shown to be simply-connected just before. Note that "T" here is for the "torsion". We can see that ${S^2}_{\rm T}$ can be chosen as a sphere smoothly embedded in  ${q_f}^{-1}(\bar{N}-{\rm Int}\ N(\partial \bar{N}))$.
It is well-known that the 2nd homotopy group of a closed, connected and orientable surface of genus $g \geq 1$ is the trivial group. We can also see that the 2nd homotopy group of ${q_f}^{-1}(\bar{N}-{\rm Int}\ N(\partial \bar{N}))$ is free and of rank $l_N$ by the homotopy exact sequence as before. More precisely, an isomorphism between the 2nd homotopy groups of ${q_f}^{-1}(\bar{N}-{\rm Int}\ N(\partial \bar{N}))$ and the base space $\bar{N}-{\rm Int}\ N(\partial \bar{N})$ is induced by the projection. Hurewicz theorem shows that the embedding of ${S^2}_{\rm T}$ must be null-homotopic. This is a contradiction. \\
 \ \\
	STEP 4 Construction of a desired map and the $m$-dimensional manifold. Our main new work. \\
	\ \\
	STEP 4-1 The structure of the desired map. \\
	\ \\
	In STEP 3, we construct a map by setting $l_N=2$ first. 
	Let $i$ be any integer greater than $1$ or $i=0$ as in the assumption.
	In addition, we set $P_1:=S^3$. We also set $P_2:=S^3$ in the case $l=2$. In the case $l>2$, we set $P_2$ as a manifold represented as a connected sum of $l^{\prime}:=l-l_N=l-2$ copies of $S^2 \times S^1$ in the case $i=0$ and one represented as a connected sum of $l^{\prime}:=l-l_N=l-2$ copies of a so-called {\it Lens space} $P_{{\rm L},i}$ enjoying the following in the case $i>1$. 
	\begin{itemize}
		\item The 1st integral homology group $H_1(P_{{\rm L},i};\mathbb{Z})$ is isomorphic to $\mathbb{Z}/i\mathbb{Z}$, the cyclic group of order $i>1$.
		\item Let an integer $i>1$ be given as before. We consider the case where the coefficient ring is $\mathbb{Z}/i\mathbb{Z}$. The 1st homology group $H_1(P_{{\rm L},i};\mathbb{Z}/i\mathbb{Z})$ is isomorphic to $\mathbb{Z}/i\mathbb{Z}$.
		\end{itemize}
	More explicitly, we can construct this as the total space of some linear bundle over $S^2$ whose fiber is diffeomorphic to $S^1$.
	
	We glue the internal smooth bundle and the boundary smooth bundle by a suitable isomorphism of the trivial bundles on the boundaries, which is represented as the product map of the identity map on the base space and some diffeomorphism on the fiber as in STEP 2.

	\ \\	
	STEP 4-2 Investigating the manifold $M$ of the domain. First we obtain a fold map and a Morse function on $M$. \\
	\ \\
	By the structure of the obtained map, we can smoothly deform the map to obtain a fold map $f^{\prime}:M \rightarrow {\mathbb{R}}^n$ enjoying the following properties. We do not explain about fundamental properties of general fold maps. However, it is important that fold maps are locally projections or represented as the products of Morse functions and the identity maps on some disks for suitable local coordinates.
	
	This also respects fundamental properties of Morse functions and fundamental arguments on some singularity theory and differential topology. We omit explaining about them. Here we can regard $x:=(x_1,\cdots,x_n) \in {\mathbb{R}}^n$ as a vector or an element in the canonically defined $n$-dimensional real vector space ${\mathbb{R}}^n$ and we consider in such ways in considerable scenes. We also apply well-known notation on vectors and vector spaces. $||x||$ is also regarded as the value of the so-called Euclidean norm at $x$ where the space is also naturally regarded as a vector space with the norm.  
	\begin{itemize}
		\item The image $f^{\prime}(M)$ is the complementary set of the interior of the disjoint union of the following two sets of the disk represented by $\{4Lx \in {\mathbb{R}}^n \mid x \in D^{n}\}
		$ for some sufficiently large integer $L>0$. 
		\begin{itemize}
			\item $\{x \in {\mathbb{R}}^n \mid ||x-(-2L,0,\cdots,0)|| \leq 2L\}$. $(-2L,0,\cdots,0)$ is also a point in $\mathbb{R} \times \{0\} \subset \mathbb{R} \times {\mathbb{R}}^{n-1}={\mathbb{R}}^n$. 
			\item $\{x \in {\mathbb{R}}^n \mid ||x-(2L,0,\cdots,0)|| \leq L\}$. $(2L,0,\cdots,0)$ is also a point in $\mathbb{R} \times \{0\} \subset \mathbb{R} \times {\mathbb{R}}^{n-1}={\mathbb{R}}^n$. 
		\end{itemize}
			\item The restriction to the singular set $S(f^{\prime})$ is a smooth embedding and the image $f^{\prime}(S(f^{\prime}))$ is represented as the disjoint union of the following four sets. They are represented as the disjoint unions of smoothly embedded standard spheres of dimension $n-1$.  
			
			\begin{itemize}
				\item $S_{f^{\prime},1}:=\{x \in {\mathbb{R}}^n \mid 2L+1 \leq ||x-(-2L,0,\cdots,0)|| \leq 2L+l^{\prime}, ||x-(-2L,0,\cdots,0)|| \in \mathbb{Z}\}$.
				\item $S_{f^{\prime},2}:=\{x \in {\mathbb{R}}^n \mid L+1 \leq ||x-(2L,0,\cdots,0)|| \leq L+l^{\prime}, ||x-(2L,0,\cdots,0)|| \in \mathbb{Z}\}$.
				\item $S_{f^{\prime},3}:=\{x \in {\mathbb{R}}^n \mid 4L-l^{\prime} \leq ||x|| \leq 4L-1, ||x|| \in \mathbb{Z}\}$.
				\item The boundary $\partial f^{\prime}(M) \subset f^{\prime}(M)$.
			\end{itemize}
		\item The preimage ${f^{\prime}}^{-1}(\{x \in {\mathbb{R}}^n \mid x-(-2L,0,\cdots,0)=(0,\cdots,0,t), t \geq 2L\})$ is diffeomorphic to $P_1$. Let $F_{P_1}$ denote this.
		\item The preimage ${f^{\prime}}^{-1}(\{x \in {\mathbb{R}}^n \mid x-(2L,0,\cdots,0)=(0,\cdots,0,t), t \geq L\})$ is diffeomorphic to $P_2$. Let $F_{P_2}$ denote this.
	\end{itemize}

\begin{figure}
	
	\includegraphics[height=45mm, width=40mm]{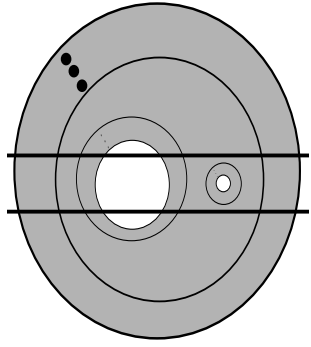}

	\caption{The images $f^{\prime}(M)$ and $f^{\prime}(S(f^{\prime}))$ and the straight lines $\{(x,-\frac{3}{2}L) \mid x \in {\mathbb{R}}^{n-1}\}$ and  $\{(x,\frac{3}{2}L) \mid x \in {\mathbb{R}}^{n-1}\}$.}
	\label{fig:1}
	
\end{figure}

We compose the fold map $f^{\prime}$ with the canonical projection mapping $x=(x_{\rm h},x_{\rm v}) \in {\mathbb{R}}^{n-1} \times \mathbb{R}={\mathbb{R}}^{n}$ to $x_{\rm v} \in \mathbb{R}$. We have a Morse function on $M$. The number of the singular point of the function is calculated as the twice the number of the connected components of the singular set $S(f^{\prime})$. In short, points regarded as poles of each sphere in the singular set $S(f^{\prime})$ contains some singular points as preimages. It is calculated as $2(3+3l^{\prime})$.

 We investigate the singular set of the Morse function.

For the function, we consider the preimages of $\{t \in \mathbb{R} \mid t \leq -\frac{3}{2}L\}$ and $\{t \in \mathbb{R} \mid t \geq \frac{3}{2}L\}$. 

By investigating the structure and the properties of the map $f^{\prime}$,
we can know that they are diffeomorphic to $S^3 \times D^{m-3}$.
This means that we do not need the $4l^{\prime}$ singular points. More precisely, we can know this by counting the points regarded as poles of spheres in $S_{f^{\prime},1}$ and $S_{f^{\prime},3}$.
The number of the singular points of a Morse function of $M$ can be reduced to $2l^{\prime}+6$. 
FIGURE \ref{fig:1} shows some information on our map $f^{\prime}$.

	\ \\	
STEP 4-3 Investigating the manifold $M$ of the domain. We apply methods respecting main ingredients of \cite{kitazawa9}. \\
\ \\
We may regard $f$ as a map whose image is same as that of $f^{\prime}$.
We respect main ingredients of \cite{kitazawa9} and the following properties and arguments are enjoyed. They are on smooth compact submanifolds with no boundaries, homology groups of $M$ and their elements (represented by these submanifolds).
First remember that $M$ is simply-connected and that the 2nd integral homology group of $M$ is free.

\begin{itemize}
	\item In the case $n \geq 4$, we have the following two.
	\begin{itemize}
		\item The 2nd integral homology group $H_2(M;\mathbb{Z})$ of $M$ is free and of rank $l^{\prime}$. In the case $i>1$, we can choose a basis of $H_2(M;\mathbb{Z}/i\mathbb{Z})$, whose rank is same as that of $H_2(M;\mathbb{Z})$. The basis is realized as the image of $H_2(F_{P_2};\mathbb{Z}/i\mathbb{Z})$ by the homomorphism induced by the inclusion of $F_{P_2}$ into $M$. In the case $i=0$, we can argue in the same way.
		\item The ($n-1$)-th integral homology group $H_{n-1}(M;\mathbb{Z})$ of $M$ is free and of rank $l_N=2$ in the case $n \geq 5$. We explain about its basis. We can take a natural basis of $H_{n-1}(\bar{N};\mathbb{Z})$. The group $H_{n-1}(\bar{N};\mathbb{Z})$ is free and of rank $2$ and we can take its basis consisting of elements represented by smoothly, disjointly and properly embedded two ($n-1$)-dimensional standard spheres (in the interior ${\rm Int}\ \bar{N}$). Our desired basis of the ($n-1$)-th integral homology group $H_{n-1}(M;\mathbb{Z})$ consists of exactly two elements represented by (the images of) some sections of the spheres in ${\rm Int}\ \bar{N}$. The ($n-1$)-th integral homology group $H_{n-1}(M;\mathbb{Z})=H_{3}(M;\mathbb{Z})$ of $M$ is free and of rank $4$ in the case $n=4$.
	\end{itemize}

By STEP 4-2 and arguments on the number of singular points of the Morse function together with the Poincar\'e duality for $M$ and universal coefficient theorem for example, this describes information on the integral homology group of $M$ completely.

We remark on the Poicar\'e dual to each element of a (suitable) basis of $H^2(M;\mathbb{Z}/i\mathbb{Z})$. We can consider a (suitable) basis of $H_2(M;\mathbb{Z}/i\mathbb{Z})$ and by considering the so-called {\it duals}, we also have a (suitable) basis of $H^2(M;\mathbb{Z}/i\mathbb{Z})$. We can consider a suitable situation and we may regard that the Poincar\'e dual to each element is an element of the homology group $H_{m-2}(M;\mathbb{Z}/i\mathbb{Z})$ represented by the total space of a subbundle of some trivial smooth bundle over some smoothly, disjointly and properly embedded two ($n-1$)-dimensional standard sphere (in the interior ${\rm Int}\ \bar{N}$). More precisely, the trivial bundle is defined by the restriction of $f$ and the fiber of each subbundle is regarded as a smoothly embedded circle in the preimage of a point with no singular points. In short, the preimage is diffeomorphic to "$F_f$ in Definition \ref{def:2}", being also a closed, connected and orientable surface.

In the case $n \geq 5$, we remark on the Poincar\'e dual to each element of a (suitable) basis of $H^{n-1}(M;\mathbb{Z})$. We can consider a (suitable) basis of $H_{n-1}(M;\mathbb{Z})$ and by considering the duals, we also have a (suitable) basis of $H^{n-1}(M;\mathbb{Z})$. We can consider a suitable situation and we may regard that the Poincar\'e dual to each element is an element represented by $F_{P_1}$ and $F_{P_2}$. In the case $n=4$, we can consider these four spheres or compact submanifolds and we can take a basis each of which is represented by each of these spheres and submanifolds.
\item In the case $n=3$, the 2nd integral homology group $H_2(M;\mathbb{Z})$ is free and has a basis consisting of the following elements. Of course the sets of the elements in these two cases are mutually disjoint.
\begin{itemize}
	\item Elements as in the case $n \geq 4$.
	\item Exactly two elements represented by (the images of) some sections of ($n-1$)-dimensional spheres smoothly embedded in ${\rm Int}\ \bar{N}$.
	 \end{itemize}
This describes information on the integral homology group of $M$ completely by STEP 4-2 and arguments on the number of singular points of the Morse function with Poincar\'e duality for $M$ and universal coefficient theorem. We can also argue in a similar way for the remaining.
\end{itemize} 

We must explain about Stiefel-Whitney classes, Pontrjagin classes and cup products.

Any $3$-dimensional compact and orientable manifold is well-known to be spin. The construction says that $M$ is spin. For
the ($m-2$)-dimensional compact submanifolds with no boundaries before are diffeomorphic to $S^{n-1} \times S^1$ and known to have trivial tangent bundles. These submanifolds are smoothly embedded in $M$ with trivial normal bundles by the structures of the map and the manifold.   

In the case $n \geq 5$, the ($n-1$)-th integral homology group $H_{n-1}(M;\mathbb{Z})$ of $M$ is free and we can take a natural basis of $H_{n-1}(M;\mathbb{Z})$, consisting of elements represented by smoothly, disjointly and properly embedded two ($n-1$)-dimensional standard spheres. In the case $n=3$, we can argue similarly. If we restrict the tangent bundle to these spheres, then they are trivial. In the case $n=4$, the $3$rd integral homology group $H_{n-1}(M;\mathbb{Z})=H_{3}(M;\mathbb{Z})$ of $M$ is free and we can take a natural basis of $H_{n-1}(M;\mathbb{Z})=H_{3}(M;\mathbb{Z})$, consisting of elements represented by smoothly, disjointly and properly embedded two $3$-dimensional standard spheres and two $3$-dimensional closed, connected and orientable manifolds represented by $F_{P_1}$ and $F_{P_2}$. $F_{P_1}$ and $F_{P_2}$ have trivial tangent bundles by the well-known fact on $3$-dimensional manifolds. The restriction of the tangent bundle of $M$ here is also trivial by the structure of the map and the manifold and we can apply this for any $n \geq 3$.

This explains about the restrictions of the tangent bundle of $M$ to the compact submanifolds by which elements of the homology groups before are represented all to see that the $a$-th Stiefel-Whitney classes and the $a$-th Pontrjagin classes are always the zero elements for any positive integer $a$. The restrictions of the tangent bundle of $M$ are trivial. Here we apply fundamental theory from linear bundles here with no related exposition. See \cite{milnorstasheff} again.

To these compact submanifolds by which the elements of the homology groups before are represented, choose two submanifolds and consider the so-called {\it generic} {\it intersection}. By the structures of the map and the manifold, they can be smoothly isotoped so that the resulting submanifolds are mutually disjoint except the presented case for the Poincar\'e duals. By the Poincar\'e duality and intersection theory for $M$, we have a desired fact on the cup products. 

By the construction, we can know the last fact, which is for the case $l>2$. This is, in short, two resulting maps are not {\it $C^{\infty}$ equivalent}. For such a notion, see \cite{golubitskyguillemin} again.
\\
\ \\
This completes the proof.
\end{proof}
\begin{proof}[A proof of Main Theorem \ref{mthm:2}]
	This immediately follows from classifications of $5$ and $6$-dimensional closed and simply-connected manifolds. See \cite{barden} again for the $5$-dimensional case and see \cite{jupp, wall1, wall2, zhubr1, zhubr2} for example for the $6$-dimensional case.
\end{proof}

\section{New problems closely related to our present study.}
 \cite{barden} shows the existence of a family of countably many $5$-dimensional closed and simply-connected manifolds each of whose 2nd integral homology group is not free, each of which is spin (not spin), and different manifolds in which are mutually non-homeomorphic. Related to Theorem \ref{thm:1} (\ref{thm:1.2}) and (\ref{thm:1.4}), according to \cite{nishioka}, such manifolds admit special generic maps into ${\mathbb{R}}^n$ if and only if $n=5$ and it is spin. Compare this to Main Theorem \ref{mthm:2}.
 
 For classifications of $6$-dimensional ones, see \cite{jupp, wall1, wall2, zhubr1, zhubr2} again for example. We can remark on the $6$-dimensional cases somewhat similarly. These $6$-dimensional cases are more complicated and we do not explain about them precisely.
 
Respecting these observations for example, we present several problems.
	\begin{Prob}
	\label{prob:1}
	Let $m>1$ be an integer.
	Is there any closed and connected (simply-connected) manifold $M$ enjoying the following properties.
	\begin{enumerate}
		\item $M$ admits no special generic maps into ${\mathbb{R}}^n$ for any integer $n$ satisfying $1 \leq n \leq m$. If we assume that $M$ is orientable and can be immersed into the one-dimensional higher Euclidean space, then we replace the condition on $n$ into $1 \leq n<m$. This is due to Eliashberg's theory \cite{eliashberg1}. \cite{eliashberg2} is also a related study.
		\item $M$ admits an (F)SGL map into some Euclidean space ${\mathbb{R}}^n$ satisfying $1 \leq n \leq m$. If we assume that $M$ is orientable and can be smoothly immersed into the one dimensional higher Euclidean space, then we replace the condition on $n$ here into $1 \leq n<m$.
	\end{enumerate}
\end{Prob}
	We have various affirmative answers.
	
	In the $2$-dimensional case, the torus $S^1 \times S^1$ and the Klein Bottle give examples. They do not admit special generic maps into $\mathbb{R}$. They admit simply generalized special generic maps into $\mathbb{R}$ and it is a Morse-Bott function. More precisely, "$F_f$" is the disjoint union of two copies of $S^1$ in Definition \ref{def:2}. We can check this as an exercise.
	
	In the case of $3$-dimensional closed, connected and orientable manifolds, all such manifolds admit FSGL maps into $\mathbb{R}$ by our proof of Theorem \ref{thm:2}. Furthermore, $S^3$, $S^1 \times S^2$ and so-called {\it Lens spaces} admit simply generalized special generic maps into $\mathbb{R}$ whereas the others do not. This is due to the theory of Heegaard splittings. 
	Theorem \ref{thm:1} (\ref{thm:1.1}) implies that most of $3$-dimensional closed, connected and orientable manifolds are desired examples.
	
According to \cite{kitazawa4, kitazawa5} in the case where the manifold can not be smoothly immersed into the one dimensional higher Euclidean space, we have an example in the case where the manifold $M$ is simply-connected.
		The $3$-dimensional complex projective space ${\mathbb{C}P}^3$ gives an example. This is realized as the total space of a linear bundle over $S^4$ whose fiber is the unit sphere $S^2$. The composition of the projection onto $S^4$ with a canonical projection into ${\mathbb{R}}^4$ is an FSGL map which is also a simply generalized special generic map. The preimage "$F_f$ in Definition \ref{def:2}" is diffeomorphic to the disjoint union of two $2$-dimensional spheres.
\begin{Prob}
		\label{prob:2}
We do not know any explicit case where the manifold is simply-connected and can be smoothly immersed into the one dimensional higher Euclidean space. Can we find?
\end{Prob}
For example, the following is also unknown.
\begin{Prob}
	\label{prob:3｝}
	Can we find some simply-connected manifold $M$ for Problem \ref{prob:1} such that the manifold $F_f$ is connected.
\end{Prob}   
	\section{Acknowledgement.}
	The author would like to thank Takahiro Yamamoto for related discussions on our previous study \cite{kitazawa9}, especially, on the terminology "({\it simply}) {\it generalized special generic maps}" and meanings of the studies in the singularity theory of differentiable maps and applications to algebraic topology and differential topology of manifolds. These discussions have motivated the author to study further and contributed to \cite{kitazawa10}, which has introduced {\it SGL} maps first in a generalized style and studies fundamental properties and non-trivial explicit construction of such maps. This leads us to our present study on manifolds admitting such maps for some very explicit and important cases. We can see that our new study applies \cite{kitazawa9, kitazawa10} and some new generalized methods.

	\end{document}